\documentclass{article}

\usepackage[english]{babel}
\usepackage{amsfonts,dsfont,mathrsfs}
\usepackage{amsmath,amssymb,amsthm,amsfonts}
\usepackage{graphicx}
\usepackage{caption}
\usepackage{xcolor} 
\usepackage{makeidx}

\usepackage{epsfig}
\usepackage{graphicx}
\usepackage{graphics}
\usepackage{float}
\usepackage{multirow}
\usepackage{color}
\usepackage{fullpage}
\usepackage[normalem]{ulem} 
\usepackage{makeidx}
\usepackage{xspace}
\usepackage{enumitem,hyperref,graphicx}
\usepackage{nameref}

\makeindex

\usepackage{lineno}

\def\namedlabel#1#2{\begingroup
    #2%
    \def\@currentlabel{#2}%
    \phantomsection\label{#1}\endgroup}

\makeatletter

\newtheorem{corollary}{Corollary}[section]
\newtheorem{Theorem}{Theorem}[section]

\newtheorem{example}{Example}[section]

\newtheorem{proposition}{Proposition}[section]

\definecolor{darkred}{rgb}{1, 0.1, 0.3}
\definecolor{darkblue}{rgb}{0.1, 0.1, 1}
\definecolor{darkgreen}{rgb}{0,0.6,0.5}

\newcommand{\R}{\mathbb{R}}
\renewcommand{\H}{\mathcal{H}}

\usepackage[colorinlistoftodos]{todonotes}
\begin{document}

\title{Proximal determination of convex functions}
 
\author{Emilio Vilches\footnote{Instituto de Ciencias de la Ingenier\'ia, Universidad de O'Higgins, email: emilio.vilches@uoh.cl}}

%\institute{Stockholm University, Stockholm, Sweden\\\email{elena.touli@math.su.se}\andthe Ohio State University Columbus, Ohio, U.S.A.\\\email{yusu@cse.ohio-state.edu}\\}\\

%\begin{document}
\maketitle
%\linenumbers
%\setcounter{page}{0}

\begin{abstract}
We provide comparison principles for convex functions through its proximal mappings. Consequently, we prove that the norm of the proximal operator determines a convex the function up to a constant. A new characterization of Lipschitzianity in terms of the proximal operator is given.
\end{abstract}

\section{Introduction}

Let $\mathcal{H}$ be a real Hilbert space with inner product $\langle\cdot,\cdot\rangle$ and norm $\Vert \cdot \Vert$. By determination of a convex function $f\colon \H \to \mathbb{R}\cup \{+\infty\}$, we mean a result of type ``if $f$ satisfies a given condition, then $f$ is uniquely determined up to constant.''  The first determination result was proved by J.J. Moreau in Hilbert spaces (see \cite[p.287]{Moreau1965}):
\begin{Theorem}[Moreau]\label{Moreau}
If $f,g\colon \Gamma_0(\H)$ are two functions such that 
$$
\operatorname{prox}_f(x)=\operatorname{prox}_g(x) \quad \textrm{ for all } x\in \H,
$$
then $f$ and $g$ differ by a constant.
\end{Theorem}
\noindent Moreau used the latter result to prove that the subgradients uniquely determine a convex function, which is known as an integration result.  Since then, several integration results appeared for convex and nonconvex functions (see, e.g., \cite{Thibault1995,Thibault2005,Thibault2010}). In this paper, by using a recent result on the determination of convex functions \cite{UOH2020}, we provide a new determination result by showing that the norm of the proximal operator  determines a convex the function up to a constant (Proposition \ref{main} and Theorem \ref{teo2}). For this, we establish comparison principles for convex functions through its proximal mapping (Theorem \ref{teo-comp}), which is also used to obtain  a new characterization of Lipschitzianity (Proposition \ref{prop-Lipschitz}).

The paper is organized as follows. After some preliminaries, in Section \ref{Comparison}, we present comparison principles for convex functions in terms of its proximal operators and a new  characterization of Lipschitz convex functions through proximal operators. These principles are the basis of the developments of Section \ref{Determination}, where it is shown that the norm of the proximal operator determines a convex function completely, up to a constant. 
%The paper ends with conclusions and final remarks.

\section{Preliminaries}

Let $\mathcal{H}$ be a real Hilbert space endowed with an inner product $\langle\cdot,\cdot\rangle$ and associated norm $\Vert \cdot \Vert$.  We denote by $\Gamma_0(\mathcal{H})$ the set of all proper, convex and lower semicontinuous functions from $\mathcal{H}$ with values in $\R\cup \{+\infty\}$. For $f\in \Gamma_0(\mathcal{H})$, its Legendre-Fenchel conjugate function $f^*:\mathcal{H}\to \R\cup \{+\infty\}$ is given by 
\[
{f^*(x^*) = \sup_{v\in \mathcal{H}}\{ \langle x^*,v\rangle - f(v)   \}.}
\]
It is known that $f^*\in\Gamma_0(\mathcal{H})$ and that for every $(x,x^*)\in \mathcal{H}\times\mathcal{H}$, the Legendre-Fenchel inequality holds, that is
\[
 f(x) + f^*(x^*) \geq \langle x^*,x\rangle.
\]  
For a closed set $C$, we denote by $\delta_{C}$ de indicator function of $C$, that is, $\delta_C(x)=0$ if $x\in C$ and $\delta_{C}(x)=+\infty$ if $x\neq C$. It is clear that $\delta_C\in \Gamma_0(\H)$ if and only if  $C$ is closed  and convex. Moreover, $(\delta_C)^{\ast}=\sigma_C$, where $\sigma_C$ is the support function of $C$ defined by $\sigma_C(x)=\sup_{y\in C}\langle y,x\rangle$. \newline
\noindent  For $\lambda>0$, the Moreau envelope of $f$ of index $\lambda$ is the function $f_{\lambda}:\mathcal{H}\to\R$ given by
\[
{f_{\lambda}(x):=\inf_{y\in\mathcal{H}} \left\{  f(y) + \frac{1}{2\lambda}\|x-y\|^2 \right\}}.
\]
{The above infimum} is attained at a unique point, $\operatorname{prox}_{\lambda f}(x)$. The mapping $\operatorname{prox}_{\lambda f}:\mathcal{H}\to\mathcal{H}$ is non-expansive and for $\lambda=1$ it is called the proximal operator, that is, 
$$
\operatorname{prox}_f(x)=\operatorname{argmin}_{y\in \H} \left\{  f(y) + \frac{1}{2}\|x-y\|^2 \right\}.
$$
It is known that $f_{\lambda}$ is convex, continuously differentiable on $\H$, and its derivative is given by
\begin{equation}\label{derivada}
\nabla f_{\lambda}(x) = \frac{1}{\lambda} (x-\operatorname{prox}_{\lambda f}(x)) \textrm{ for all } x\in \H
\end{equation}
Moreover, 
\begin{equation}\label{conjugada-Moreau}
(f_{\lambda})^{\ast}(x)=f^{\ast}(x)+\frac{\lambda}{2}\Vert x\Vert^2 \textrm{ for all } x\in \H.
\end{equation}
We refer to  \cite{attouch2014variational,BC2017} for more details of Moreau envelope and its applications. \newline

\noindent To obtain our results, we need the Moreau decomposition (see  \cite[p. 280]{Moreau1965}).
\begin{proposition}[Moreau decomposition]\label{Moreau-deco} If $f\in \Gamma_0(\H)$, then
$$
\operatorname{prox}_{ f}(x)+ \operatorname{prox}_{f^{\ast}}(x)=x \quad \textrm{ for all } x\in \H.
$$
\end{proposition}

We end this section with a comparison principle for convex functions through its gradients (see \cite[Theorem~3.1]{UOH2020}).  This principle is the basis for the determination of convex functions through the norm of (sub)gradients. We refer to \cite{UOH2020} for further results in this direction.
\begin{proposition}\label{Gradiente}
Let $f, g\in \Gamma_0(\H)$ be two G\^{a}teaux differentiable convex functions bounded from below such that 
$$
\Vert \nabla f(x)\Vert\leq \Vert \nabla g(x)\Vert \quad \textrm{ for all } x\in \H.
$$
Then,   $f-\inf f\leq  g-\inf g$.
\end{proposition}

\section{Comparison principles}\label{Comparison}    

The following result is a comparison principle for convex functions.
\begin{Theorem}\label{teo-comp}
Let $f,g\colon \Gamma_0(\H)$ be two functions such for some $x_0\in \operatorname{dom}f\cap \operatorname{dom}g$ and 
$$
\Vert \operatorname{prox}_f(x)-x_0\Vert\leq \Vert \operatorname{prox}_g(x)-x_0\Vert \quad \textrm{ for all } x\in \H.
$$
Then, $g-g(x_0)\leq f-f(x_0)$.
\end{Theorem}
\begin{proof} By virtue of Legendre-Fenchel inequality, for all $x, u\in \H$
\begin{equation*}
f^{\ast}(x)+f(u)\geq \langle x,u\rangle \textrm{ and } g^{\ast}(x)+g(u)\geq \langle x,u\rangle.
\end{equation*}
Thus, if $x_0\in \operatorname{dom}f\cap \operatorname{dom}g$, then 
\begin{equation*}
f^{\ast}(x)-\langle x,x_0\rangle \geq -f(x_0) \textrm{ and } g^{\ast}(x)-\langle x,x_0\rangle\geq -g(x_0),
\end{equation*}
which implies that the maps $x\mapsto f^{\ast}(x)-\langle x_0,x \rangle$ and $x\mapsto g^{\ast}(x)-\langle x_0,x \rangle$ are bounded from below. \newline
\noindent For $\lambda=1$,  let us consider 
$$
\tilde{f}=\left(f^{\ast}-\langle x_0,\cdot\rangle \right)_{\lambda} \textrm{ and } \tilde{g}=\left(g^{\ast}-\langle x_0,\cdot\rangle \right)_{\lambda}.
$$
Then, $\tilde{f}$ and $\tilde{g}$ are $C^{1,1}$  and bounded from below functions with 
$$
\nabla \tilde{f}(x)=x-\operatorname{prox}_{f^{\ast}-\langle x_0,\cdot\rangle}(x) \textrm{ and } \nabla \tilde{g}(x)=x-\operatorname{prox}_{g^{\ast}-\langle x_0,\cdot\rangle}(x).
$$
Moreover, according to Moreau's decomposition and properties of the proximal operator, for all $x\in \H$
\begin{equation*}
\begin{aligned}
\nabla \tilde{f}(x)&=x-\operatorname{prox}_{f^{\ast}-\langle x_0,\cdot\rangle}(x)=\operatorname{prox}_{(f^{\ast}-\langle x_0,\cdot\rangle)^{\ast}}(x)=\operatorname{prox}_{f(\cdot+x_0)}(x)=\operatorname{prox}_f(x+x_0)-x_0,\\
\nabla \tilde{g}(x)&=x-\operatorname{prox}_{g^{\ast}-\langle x_0,\cdot\rangle}(x)=\operatorname{prox}_{(g^{\ast}-\langle x_0,\cdot\rangle)^{\ast}}(x)=\operatorname{prox}_{g(\cdot+x_0)}(x)=\operatorname{prox}_g(x+x_0)-x_0.
\end{aligned}
\end{equation*}
Therefore, for all $x\in \H$
$$
\Vert \nabla \tilde{f}(x)\Vert \leq \Vert \nabla \tilde{g}(x)\Vert.
$$
Hence, by virtue of Proposition \ref{Gradiente}, $$\tilde{f}\leq \tilde{g}+ \inf \tilde{f}-\inf \tilde{g}=\tilde{g}-f(x_0)+g(x_0),$$
where we have used that $$ \inf \tilde{f}=\inf (f^{\ast}-\langle x_0,\cdot\rangle)=-f^{\ast\ast}(x_0)=-f(x_0) \textrm{ and }  \inf \tilde{g}=\inf (g^{\ast}-\langle x_0,\cdot\rangle)=-g^{\ast\ast}(x_0)=-g(x_0).$$ Then, by conjugation, we obtain that 
$$
(\tilde{g})^{\ast} \leq (\tilde{f})^{\ast}-f(x_0)+g(x_0).
$$
Then,  due to \eqref{conjugada-Moreau},
\begin{equation*}
\begin{aligned}
(\tilde{f})^{\ast}(x)&=(f^{\ast}-\langle x_0,\cdot\rangle)^{\ast}(x)+\frac{1}{2}\Vert x\Vert^2=f(x+x_0)+\frac{1}{2}\Vert x\Vert^2,\\
(\tilde{g})^{\ast}(x)&=(g^{\ast}(x)-\langle x_0,\cdot\rangle)^{\ast}+\frac{1}{2}\Vert x\Vert^2=g(x+x_0)+\frac{1}{2}\Vert x\Vert^2.
\end{aligned}
\end{equation*}
Hence,  $$g(x+x_0)\leq f(x+x_0)-f(x_0)+g(x_0),$$ 
which ends the proof.
\end{proof}
%\begin{example} \EVN{Me falta demostrar que Csubset E entonces blabla}Let $C, D$ two closed and convex sets with $0\in C\cap D$. Then, 
%if
%$$
%\Vert \operatorname{proj}_C(x)\Vert \leq \Vert \operatorname{proj}_D(x)\Vert \textrm{ for all } x\in \H,
%$$
%then, $C\subset D$.
%\end{example}
%With the same arguments as the latter theorem, we obtain the following variation of Theorem \ref{teo-comp}.
%\begin{Theorem}\label{main2}
%\EVN{Es consecuencia directa del otro, poniendo 0 en el dominio, pues 0 esta en el dom ssi la conjugada es acotada.}Let $f,g\colon \Gamma_0(\H)$ be two functions such that $f^{\ast}$ and $g^{\ast}$ are bounded from below and 
%$$
%\Vert \operatorname{prox}_f(x)\Vert\leq \Vert \operatorname{prox}_g(x)\Vert \quad \textrm{ for all } x\in \H.
%$$
%Then,  $g-g(0)\leq  f-f(0)$. 
%\end{Theorem}
%\begin{proof} For $\lambda=1$, consider $\tilde{f}=(f^{\ast})_{\lambda}$ and $\tilde{g}=(g^{\ast})_{\lambda}$ and repeat the proof of Theorem \ref{teo-comp}.
%\end{proof}
%\begin{remark}
%From the Legendre-Fenchel inequality, it follows that the condition $0\in \operatorname{dom}f$ implies that $f^{\ast}$ is bounded from below. 
%\end{remark}

The following proposition provides an example of application of Theorem \ref{teo-comp}.
\begin{proposition} Let $\ell\geq 0$ and $g\in \Gamma_0(\H)$ such that $g^{\ast}$ is bounded from below and 
$$
\Vert x\Vert-\ell  \leq \Vert \operatorname{prox}_g(x)\Vert \quad \textrm{ for all } x\in \H.
$$
Then, $g-g(0)\leq \ell \Vert \cdot \Vert$. Moreover, if $\ell\equiv0$, then $g$ is constant.
\end{proposition}
\begin{proof} Indeed, if  $f=\ell \Vert \cdot\Vert$, then $$\operatorname{prox}_f(x)=\left(1-\frac{\ell}{\max\{ \Vert x\Vert,\ell\}}\right)x,$$ and $f^{\ast}$ is bounded from below with $\inf f^{\ast}=0$. Thus, by Theorem \ref{teo-comp}, $g-g(0)\leq \ell \Vert \cdot \Vert$. Finally, if $\ell=0$, then $g$ is a constant function (a convex function which is bounded from above is constant).
\end{proof}

The following result gives a Lipschitzianity characterization for a convex function.
\begin{proposition}\label{prop-Lipschitz}
Let $f\colon \H \to \mathbb{R}$ be a convex and lower semicontinuous function. Then, $f$ is $\ell$-Lipschitz if and only 
\begin{equation}\label{Lipschitz}
\Vert x\Vert -\ell \leq \Vert \operatorname{prox}_f(x+y)-y\Vert \textrm{ for all } x,y\in \H.
\end{equation}
\end{proposition}
\begin{proof}
On the one hand, if $f$ is $\ell$-Lipschitz, then for all $x,y$
$$
\Vert x\Vert -\ell \leq \Vert x+y-\operatorname{prox}_f(x+y)\Vert +\Vert \operatorname{prox}_f(x+y)-y\Vert \leq \ell+\Vert \operatorname{prox}_f(x+y)-y\Vert,
$$
where we have used that $x+y-\operatorname{prox}_f(x+y)\in \partial f(\operatorname{prox}_f(x+y)) \subset \ell\, \mathbb{B}$. \newline
\noindent On the other hand, assume that \eqref{Lipschitz} holds and fix $y\in \H$. Let us consider the functions $h:=f(\cdot+y)$ and $g=\ell \Vert \cdot\Vert$. Then, for all $x\in \H$
$$
\operatorname{prox}_h(x)=\operatorname{prox}_f(x+y)-y \textrm{ and } \operatorname{prox}_g(x)=\left( 1-\frac{\ell}{\max\{ \Vert x\Vert,\ell\}}\right)x.
$$
Moreover, since $\operatorname{dom}(f)=\H$,  $h^{\ast}$ is bounded from below and $\inf h^{\ast}=-f(y)$. 
Therefore, for all $x\in \H$
\begin{equation}
\Vert \operatorname{prox}_g(x)\Vert \leq \Vert x\Vert -\ell \leq \Vert \operatorname{prox}_f(x+y)-y\Vert =\Vert \operatorname{prox}_h(x)\Vert.
\end{equation}
By virtue of Theorem \ref{teo-comp}, we obtain that 
$$
f(x+y)\leq \ell \Vert x\Vert +\inf g^{\ast}-\inf h^{\ast}.
$$
Finally, since $\inf g^{\ast}=0$ and $\inf h^{\ast}=-f(y)$, we get that 
$$
f(x+y)\leq f(y)+\ell \Vert x\Vert,
$$
which implies that $f$ is $\ell$-Lipschitz.

\end{proof}

\section{Determination of convex functions}\label{Determination}

Since then, several integration results appeared In this section, we present the main finding of the paper; that is, the norm of the proximal operator determines a convex function up to a constant. The following two results extends Theorem \ref{Moreau}. 

\begin{proposition}\label{main}
Let $f,g\colon \Gamma_0(\H)$ be two functions such that for some $x_0\in \operatorname{dom}f\cap \operatorname{dom}g$ 
$$
\Vert \operatorname{prox}_f(x)-x_0\Vert= \Vert \operatorname{prox}_g(x)-x_0\Vert \quad \textrm{ for all } x\in \H.
$$
Then, $f-f(x_0)=g-g(x_0)$.
\end{proposition}
The following result summarizes several determination principles for convex functions. 
\begin{Theorem}\label{teo2}
Let $f,g\colon \Gamma_0(\H)$ be two functions such that $f^{\ast}$ and $g^{\ast}$ are bounded from below. Then,  the following assertions are equivalent:

\begin{itemize}
\item[(i)] For all $x\in \H$, $\Vert \operatorname{prox}_f(x)\Vert= \Vert \operatorname{prox}_g(x)\Vert$.
\item[(ii)] For all $x\in \H$, $f(x)=g(x)+ \inf f^{\ast}-\inf g^{\ast}$.
\item[(iii)] For all $x\in \H$, $\partial f(x)^{\circ}=\partial g(x)^{\circ}$, where $\partial f(x)^{\circ}=\operatorname{Proj}_{\partial f(x)}(0)$ and $\partial g(x)^{\circ}=\operatorname{Proj}_{\partial g(x)}(0)$.
\item[(iv)] For all $x\in \H$, $\partial f(x)=\partial g(x)$.
\item[(v)] For all $x\in \H$, $ \operatorname{prox}_f(x)= \operatorname{prox}_g(x)$.
\end{itemize}
\end{Theorem}
\begin{proof}
(i)$\Rightarrow$(ii) follows from Theorem \ref{teo-comp}.  (ii)$\Rightarrow$(iii) is trivial. 
(iii)$\Rightarrow$ (iv) follows from \cite[Corollaire~2.2]{Brezis1973}. (iv)$\Rightarrow$(v) follows from the formula $\operatorname{prox}_f(\cdot)=(I+\partial f)^{-1}(\cdot)$. Finally, (v)$\Rightarrow$(i) is trivial.
\end{proof}
The following example shows that the hypotheses for  the implication $(i)\Rightarrow(ii)$ are sharp.
\begin{example}
Let us consider $f=\delta_{\{x_1\}}$ and $g=\delta_{\{x_2\}}$, where $x_1\neq x_2$ and $\Vert x_1\Vert=\Vert x_2\Vert$. Then, for all $x\in \H$
$$\Vert \operatorname{prox}_f(x)\Vert =\Vert x_1\Vert =\Vert x_2\Vert =\Vert \operatorname{prox}_g(x)\Vert.$$
However, $f^{\ast}=\langle x_1,\cdot\rangle$ and $g^{\ast}=\langle x_2,\cdot\rangle$ are not bounded from below.
\end{example}

Theorem \ref{teo2}  allow us to obtain the following characterization of support functions.
\begin{corollary}
Let $C$ be a nonempty, closed and convex set containing $0$.  Then, $f\in \Gamma_0(\H)$ satisfies 
\begin{equation}\label{distance}
\Vert \operatorname{prox}_f(x)\Vert =d_{C}(x) \textrm{ for all } x\in \H
\end{equation}
if and only if $f$ is the support of $C$ up to a constant. 
\end{corollary}
\begin{proof}
Indeed, on the one hand, if $f$ is the support of $C$ up to a constant, then $\operatorname{prox}_f=\operatorname{prox}_{\sigma_C}$, which implies \eqref{distance}. On the other hand, if \eqref{distance} holds, then $f^{\ast}$ is bounded from below. Moreover, by Proposition \ref{Moreau-deco},
$$
\Vert \operatorname{prox}_f(x)\Vert=d_{C}(x)=\Vert x-\operatorname{proj}_C(x)\Vert =\Vert \operatorname{prox}_{\sigma_{C}}(x)\Vert \textrm{ for all } x\in \H,
$$
where we have used that $(\delta_{C})^{\ast}=\sigma_{C}$. Therefore, by Theorem \ref{teo2}, $f$ is the support of $C$ up to a constant.
\end{proof}

\paragraph{Acknowledgements} The author wishes to express his gratitude to Bao Tran Nguyen,  Pedro P\'erez-Aros and David Salas  from Universidad de OÕHiggins and Lionel Thibault  from University of Montpellier for their valuable comments about the presentation of the article. The author was funded by ANID Chile under grants Fondecyt de Iniciaci\'on No. 11180098 and Fondecyt Regular No. 1200283.

\end{document}